\documentclass{amsart}
\usepackage{latexsym}
\usepackage{amsmath}
\usepackage{amssymb}
\usepackage[all]{xy}

\newtheorem{thm}{Theorem}[section]
\newtheorem{prop}[thm]{Proposition}
\newtheorem{cor}[thm]{Corollary}
\newtheorem{lem}[thm]{Lemma}

\theoremstyle{definition}
\newtheorem{defn}[thm]{Definition}

\newtheorem{assertion}[thm]{Assertion}
\theoremstyle{remark}
\newtheorem{rem}[thm]{Remark}

\newcommand{\F}{{\mathcal F}}
\newcommand{\C}{{\mathcal C}}
\newcommand{\calO}{{\mathcal O}}
\newcommand{\D}{{\mathcal D}}
\newcommand{\calL}{\mathcal L}
\newcommand{\R}{{\mathbb R}}
\newcommand{\m}{\mathfrak{m}}
\newcommand{\Lfb}{\mathsf{Lfb}}
\newcommand{\Fgp}{\mathsf{Fgp}}
\newcommand{\Mod}{\mathsf{Mod}}
\newcommand{\VB}{\mathsf{VBb}_{(S,\C)}}
\newcommand{\mapright}[1]{%
 \smash{\mathop{%
  \hbox to 1cm{\rightarrowfill}}\limits_{#1}}}
\newcommand{\maprightd}[2]{%
 \smash{\mathop{%
  \hbox to 1.2cm{\rightarrowfill}}\limits^{#1}\limits_{#2}}}
\newcommand{\mapleft}[1]{%
 \smash{\mathop{%
  \hbox to 1cm{\leftarrowfill}}\limits_{#1}}}
\newcommand{\mapleftu}[1]{%
 \smash{\mathop{%
  \hbox to 0.8cm{\leftarrowfill}}\limits^{#1}}}
\newcommand{\maprightu}[1]{%
 \smash{\mathop{%
  \hbox to 1cm{\rightarrowfill}}\limits^{#1}}}
\newcommand{\maprightud}[2]{%
 \smash{\mathop{%
  \hbox to 1cm{\rightarrowfill}}\limits^{#1}_{#2}}}
\newcommand{\mapleftud}[2]{%
 \smash{\mathop{%
  \hbox to 1cm{\leftarrowfill}}\limits^{#1}_{#2}}}
\newcounter{eqn}[section]

\def\theeqn{\textnormal{(\thesection.\arabic{eqn})}}

\def\eqnlabel#1{%
  \refstepcounter{eqn}%
  \label{#1}%
  \leqno{\theeqn}}

\begin{document}

\title{On the category of stratifolds 
}
\footnote[0]{

This research was  partially supported by a Grant-in-Aid for challenging Exploratory Research 16K13753 from Japan Society for the Promotion of Science.

Department of Mathematical Sciences, 
Faculty of Science,  
Shinshu University,   
Matsumoto, Nagano 390-8621, Japan   
e-mail:{\tt kuri@math.shinshu-u.ac.jp}
}

\author{Toshiki Aoki and Katsuhiko KURIBAYASHI}
\date{}
   
\maketitle

\begin{abstract} Stratifolds are considered from a categorical point of view.  
We show among others that the category of stratifolds fully faithfully embeds into the category of 
${\mathbb R}$-algebras as does the category of smooth manifolds.  
We prove that a variant of the  Serre-Swan theorem holds for stratifolds. In particular, 
the category of vector bundles over a stratifold is shown to be equivalent to the category of vector bundles 
over an associated affine scheme although the latter is in general larger than the stratifold itself.   
\end{abstract}

\section{Introduction}

Stratifolds have been introduced by Kreck \cite{Kreck}
\footnote{The review article \cite{BK} of Kloeckner contains a nice historical review of different notions of stratified space.}. 
The new notion subsumes manifolds and algebraic varieties with isolated singularities 
as examples; see \cite{G}.  
One of its advantages is that stratifolds give geometric counterparts of singular homology classes of  
a CW complex in much the same way as manifolds give geometric homology classes in the sense of Jakob \cite{Jakob}. More precisely, 
such a homology class is represented by an appropriate bordism class of stratifolds.
One might therefore expect that stratifolds share some of the fascinating properties of manifolds and varieties.  
In this article, we focus on such properties for stratifolds and investigate them from 
categorical and sheaf-theoretical points of view. 

Pursell \cite{P} showed that the category of smooth manifolds fully faithfully embeds into the category of $\R$-algebras. 
We extend this result to the category of stratifolds. 

\begin{thm}\label{thm:main}
The category of stratifolds fully faithfully embeds 
into the category of ${\mathbb R}$-algebras. 
\end{thm}

A stratifold $(S, \C)$ consists of a topological space $S$ 
and a subalgebra $\C$ of the $\mathbb{R}$-algebra of continuous real-valued functions on the 
underlying space. Such a subalgebra defines a ringed space 
which is called the {\it structure sheaf} of the stratifold. Indeed, the subalgebra is 
nothing but the algebra of global sections of the sheaf. 
The assignment of the algebra to a stratifold, namely the forgetful functor $F$ 
defined by $F(S, \C) = \C$, gives rise to the embedding in Theorem \ref{thm:main}.  

Let $M$ be a smooth manifold. Then the prime spectrum of the ring $C^\infty(M)$ of real-valued smooth functions with the Zariski topology 
is larger than the underlying space $M$ in general. However, the {\it real spectrum}, which is a subspace of the prime spectrum, 
is homeomorphic to $M$. This fact is shown to extend to stratifolds; see Propositions \ref{prop:Spec} and 
\ref{prop:sheaf}. 

The results mentioned above lead us naturally to considering the affine scheme of the global sections of 
the structure sheaf of a stratifold. 
In consequence, we see that the restriction of the affine scheme to the real spectrum 
is isomorphic to a given stratifold as a ringed space; 
see Theorem \ref{thm:affine_scheme}.  We are convinced that the result, a sheaf-theoretic description of a stratifold, 
enables one to consider stratifolds in the framework of derived differential geometry \cite{J, Spivak} though this issue is 
not pursued in this manuscript; see Remark \ref{rem:C^infty}. 

The category of vector bundles over a smooth manifold $M$ is equivalent to the category of finitely generated projective modules over 
$C^\infty(M)$ by a classical result of Swan \cite{Swan}. An analogous result for algebraic varieties has been obtained by Serre \cite{Serre}. 
It is thus worthwhile to investigate a Serre-Swan type theorem for stratifolds. 
To this end, we introduce the appropriate notion of vector bundle over stratifolds; 
see Definition \ref{defn:bundles} and Proposition \ref{prop:stratifold_bundles}. 
With our definition, we get the following result; see Theorem \ref{thm:3} for the precise statement.  

\begin{thm}
The Serre-Swan theorem holds for stratifolds. 
\end{thm}

As a consequence, we deduce that the category of vector bundles over a stratifold is equivalent to 
that of vector bundles over the affine scheme associated to the stratifold though the underlying prime spectrum 
is larger than the stratifold itself in general; see Remark \ref{rem:comparison}. 

The rest of this article is organized as follows. 
In Section \ref{section:Proof}, after recalling the definition of a stratifold and its important properties, 
we prove Theorem \ref{thm:main}. 
We investigate stratifolds and their category from a sheaf-theoretical point of view in Section 3. 
In Section 4, the notion of {\it vector bundle} over a stratifold is introduced and the Serre-Swan theorem 
is shown to hold for any stratifold.  
In Section 5, we characterize morphisms of stratifolds by local data, and describe them inside the category of 
{\it diffeological spaces}; see \cite{S, IZ}. 
In Section 6, we study  the product of stratifolds from a categorical perspective. 
This is used in Section 4 in the course of proving the Serre-Swan theorem. 

We conclude this section with comments. An important device in the study of stratifolds is the existence of so-called 
{\it local retractions} near each point of the stratifolds; see \cite{Kreck}. 
These retractions are essential at several places in this article; see Sections 4, 5 and 6. 
Some of proofs in Sections 2 and 3 are straightforward. Yet, they are instructive for clarifying what properties of manifolds and stratifolds are 
responsible for the obtained results. These results are needed to set up a framework for describing the Serre-Swan theorem in our 
context. One of highlights in this manuscript is that a version of the Serre-Swan theorem for stratifolds is proved 
without using tautological bundles or the Whitney immersion theorem as is usually done for proving 
the theorem in the case of a manifold.

\section{The real spectrum of a stratifold}
\label{section:Proof}

This section contains a brief review of stratifolds. We begin with the definition of a differential space in the sense of Sikorski \cite{Sik}. 

\begin{defn} \label{defn:differential_space}
A {\it  differential space} is a pair $(S, \C)$ consisting of a topological space $S$ and an $\mathbb{R}$-subalgebra $\C$
of the $\mathbb{R}$-algebra $C^0(S)$ of continuous real-valued functions on $S$, which is supposed to be {\it locally detectable} and 
$C^\infty$-{\it closed}.  

\medskip
Local detectability means that $f \in \C$ if and only if 
for any $x \in S$, there exist an open neighborhood $U$ of $x$ and an element $g \in \C$ such that 
$f|_U = g|_U$.  

\medskip
$C^\infty$-closedness means that for each $n\geq 1$, each $n$-tuple $(f_1, ..., f_n)$ of maps in $\C$ and each smooth map 
$g : \mathbb{R}^n \to \mathbb{R}$, the composite  
$h : S \to \mathbb{R}$ defined by $h(x) = g(f_1(x), ...., f_n(x))$ belongs to $\C$.
\end{defn}

Let $(S, \C)$ and $(S', \C')$ be differential spaces. 
We call a continuous map $f: S \to S'$ a {\it morphism of the differential spaces}, denoted $f : (S, \C) \to (S', \C')$, 
if $f$ induces a map $f^* : \C' \to \C$; that is,  $\varphi \circ f \in \C$ for each $\varphi \in \C'$.  
Thus we define a category $\mathsf{Diff}$ of differential spaces. Let $\mathsf{Mfd}$ denote the category of smooth manifolds. It is readily seen that the functor $i : \mathsf{Mfd} \to \mathsf{Diff}$ defined by $i(M) = (M, C^\infty(M))$ is a fully faithful embedding. 

For any smooth paracompact manifold $M$, the defining subalgebra $C^\infty(M)$ of $C^0(M)$ has two additional properties: 
\begin{enumerate}
\item[(i)] It extends to a sheaf of $\mathbb{R}$-algebras $U \mapsto  C^\infty(U)$. 
\item[(ii)] For any open cover ${\mathcal U}$ of $M$, there exists a {\it smooth} partition of unity subordinate to ${\mathcal U}$.
In particular, the sheaf $C^\infty$ is generated by global sections in the sense that the canonical map $C^\infty(M) \to (C^\infty)_x$ is surjective 
for any $x \in M$, where  $(C^\infty)_x$ denotes the $\mathbb{R}$-algebra of the germs at $x$. 
\end{enumerate}
This in turn implies that $C^\infty(U)$ can be recovered from $C^\infty(M)$ as the set of 
{\it locally extendable functions} on $U$. With this in mind, we introduce such functions in the context of differential spaces. 

For a differential space $(S, \C)$ and a subspace $Y$ of $S$, we call an element $g \in C^0(Y)$ a {\it locally extendable function} 
if for any $x \in Y$, there exists an open neighborhood $V$ of $x$ in $Y$ and $h\in \C$ such that $g|_V = h|_V$. Let $\C_Y$ be the 
subalgebra of $C^0(Y)$ consisting of locally extendable functions on $Y$. Then it follows that $(Y, \C_Y)$ is a differential space; 
see \cite[page 8]{Kreck}. 
Thus any subspace of a differential space  
inherits the structure of a differential space.

Let $(S, \C)$ be a differential space and $x \in S$. The vector space consisting of derivations on the $\mathbb{R}$-algebra 
$\C_x$ of the germs at $x$ is denoted by $T_xS$, which is called the {\it tangent space} 
of the differential space  at $x$; see \cite[Chapter 1, section 3]{Kreck}. 

\begin{defn} \label{defn:stratifold} 
A {\it stratifold} is a differential space $(S, \C)$ such that the following four conditions hold: 

\begin{enumerate}
\item
$S$ is a locally compact Hausdorff space with countable basis; 
\item
the {\it skeleta} $sk_k(S):= \{x \in S \mid \text{dim } \!T_xS\leq k\}$ are closed in $S$;
\item
for each $x \in S$ and open neighborhood $U$ of $x$ in $S$, 
there exists a {\it bump function} at $x$ subordinate to $U$; that is, a non-negative function $\rho \in \C$ such that 
$\rho(x)\neq 0$ and such that the support $\text{supp }\!\rho :=\overline{\{p \in S \mid \rho(p) \neq 0\}}$ is contained in $U$; 
\item
the {\it strata} $S^k := sk_k(S) - sk_{k-1}(S)$ are $k$-dimensional smooth manifolds such that restriction along 
$i : S^k \hookrightarrow S$ induces an isomorphism of stalks 
$$
i^* : \C_x \stackrel{\cong}{\to} C^\infty(S^k)_x.
$$
for each $x \in S^k$.
\end{enumerate}
\end{defn}

A stratifold is {\it finite-dimensional} if there is a non-negative integer $n$ such that $S = sk_n(S)$.  In particular, the tangent spaces of 
a finite-dimensional stratifold are finite-dimensional. 

In what follows, we assume that all stratifolds are finite-dimensional. 
We may simply write $S$ for a stratifold or differential space $(S, \C)$ if no confusion arises. 
A smooth manifold $(M, C^\infty(M))$ is a typical example of a stratifold. We define the category $\mathsf{Stfd}$ of stratifolds 
as the full subcategory of $\mathsf{Diff}$ spanned by the stratifolds. Observe that the embedding $\mathsf{Mfd} \to \mathsf{Diff}$ 
mentioned above factors through $\mathsf{Stfd}$.

We here recall important properties of a stratifold.

\begin{rem}\label{rem:substratifold}
Let $(S, \C)$ be a stratifold with strata $\{S^i\}$ . \\
{(i)} Let $U$ be an open subset of $S$ and $\C_U$ the subalgebra of $\C$ consisting of locally extendable functions of $C^0(U)$ 
in $\C$. Then $(U, \C_U)$ is a stratifold with strata $\{S^i\cap U\}$; see \cite[Example 5, page 22]{Kreck}.  \\
{(ii)} For any $x \in S^i$, there exist an open neighborhood $U$ of $x$ in $S$ and a morphism 
$$r_x : (U, \C_U)  \to (U \cap S^i, \C_{U\cap S^i})$$ such that $r_x|_{U\cap S^i}= id$. Such a map is called 
a {\it  local retraction} near $x$; see  \cite[Proposition 2.1]{Kreck} \\
{(iii)} Any locally compact Hausdorff space with countable basis is paracompact and in particular countable at infinity. 
This together with the other properties of a stratifold $(S, \C)$ shows that for any open cover $\mathcal{U}$ of $S$, 
there exists a partition of unity subordinate to  $\mathcal{U}$ {\it consisting 
of functions in} $\C$, i.e. the structure sheaf ${\mathcal O}_S$ of the stratifold $(S, \C)$ is {\it fine}; 
see \cite[Proposition 2.3]{Kreck} and Sections 3 and 4. 
\end{rem}


We refer the reader to the book \cite{Kreck} of Kreck for other fundamental properties, examples of stratifolds and 
fascinating results on the stratifold (co)homology.

For an $\R$-algebra $\F$,  we define $| \F |$ to be the set of all morphisms of $\R$-algebras from $\F$ to $\mathbb{R}$ 
which preserve the unit.   
Moreover, we define a map $\widetilde{f} : | \F | \to \R$ by $\widetilde{f}(x) = x(f)$ for any $f \in \F$. 
Let $\widetilde{\F}$ be the $\R$-algebra of maps from $| \F |$ to $\R$ of the form $\widetilde{f}$ for $f \in \F$. 
Then we consider the Gelfand topology on $| \F |$; that is, $| \F |$ is regarded as the topological space with the open basis 
$$\{\widetilde{f}^{-1}(U) \mid U : \text{open in} \ {\mathbb R}, \widetilde{f} \in \widetilde{\F}\}; $$
see \cite[2.1]{GS} and \cite[3.12]{N}.
Thus the assignment of a topological space to an ${\mathbb R}$-algebra gives rise to a contravariant functor 
$$| \ \ | : {\mathbb R}\text{-}\mathsf{Alg} \to \mathsf{Top}$$
which is called the {\it realization functor}, where  ${\mathbb R}\text{-}\mathsf{Alg}$ denotes the category of ${\mathbb R}$-algebras. 

By definition, 
the map $\tau : \F \to \widetilde{\F}$ defined by $\tau(f) = \widetilde{f}$ is surjective. 
It follows that $\tau$ is an isomorphism if $\F$ is a subalgebra of the $\R$-algebra 
of continuous functions on a space; see \cite[3.14]{N}.


\begin{lem}\label{lem:2.1}
Let $(S, \C)$ be a stratifold. Then the map $\theta : S \to |\C|$ defined by 
$\theta (p)(f) = f(p)$ is a homeomorphism. 
\end{lem}

\begin{proof} 
In virtue of Remark \ref{rem:substratifold}, the same argument as in the proof of \cite[Theorem 7.2]{N} 
shows that $\theta$ is a bijection. 


For any open set $U$ in $\R$ and $f\in \C$, we see that $\theta^{-1}(\widetilde{f}^{-1}(U))= f^{-1}(U)$ since $\widetilde{f}\circ \theta = f$. 
This implies that $\theta$ is continuous. 

Let $W$ be an open set of $S$.  
By definition, a stratifold has a bump function for each $x \in W$; that is, there exists a non-negative function 
$f_x \in \C$ such that $\text{supp} f_x \subset W$ and $f_x(x) \neq 0$. Then we see that $x \in f_x^{-1}(\R^+) \subset W$ for any $x \in W$ and hence $W = \bigcup_{x \in W}f_x^{-1}(\R^+)$. Therefore, it follows that 
$$\theta(W) = \theta (\bigcup_{x \in W}f_x^{-1}(\R^+)) = \bigcup_{x \in W}(\theta^{-1})^{-1}f_x^{-1}(\R^+)
= \bigcup_{x \in W}\widetilde{f_x}^{-1}(\R^+).$$ 
Observe that 
$\widetilde{f_x}\circ \theta = f_x$ as mentioned above. 
This shows that $\theta$ is open. 
\end{proof}

Let $\F$ be a subalgebra of $C^0(X)$ the $\R$-algebra of continuous maps 
from a space $X$ to $\R$. We call  the pair $(X, \F)$ a {\it continuous space}.   
Let $\mathsf{Csp}$ be the category of continuous spaces. Observe that 
a morphism $\varphi : (X, \F_X) \to (Y, \F_Y)$ is   
a continuous map $\varphi : X \to Y$ which satisfies the condition that 
$f\circ \varphi \in \F_X$ for any $f \in \F_Y$.  
By definition, the category $\mathsf{Diff}$ of differential spaces is a full subcategory of $\mathsf{Csp}$; therefore, the categories 
$\mathsf{Mfd}$ and $\mathsf{Stfd}$ are full subcategories of $\mathsf{Csp}$ as well.   

\begin{prop} \label{prop:spectrum}
The map $\theta : S \to |\C|$ gives rise to an isomorphism 
$\theta : (S, \C) \to (|\C|, \widetilde{\C})$ of continuous spaces. 
\end{prop}

\begin{proof}
Recall the isomorphism $\tau : \C \to \widetilde{\C}$. 
We consider the composite $\theta^*\circ \tau : \C \to \widetilde{\C} \to \C$. Then 
it is readily seen  that $(\theta^*\circ \tau) (f) = f$ for any  $f \in \C$. This implies that $\theta^* : \widetilde{\C} \to \C$ 
is a well-defined isomorphism. Since $(\theta^{-1})^*\theta^* (\widetilde{f}) =\widetilde{f}$, 
it follows that $(\theta^{-1})^* : \C \to C^0(|\C|)$ factors through the subalgebra 
$\widetilde{\C}$ and that $(\theta^{-1})^* : \C \to \widetilde{\C}$ is an isomorphism. This completes the proof.  
\end{proof}

We call a maximal ideal ${\sf m}$ of $\C$ real if the quotient $\C/{\sf m}$ is isomorphic to $\R$ as an $\R$-algebra. 
Let $\text{Spec}_r\ \C$ be the {\it real spectrum}, namely the subset of the prime spectrum $\text{Spec} \ \C$ of $\C$ consisting of 
real ideals.  
We consider $\text{Spec}_r\ \C$ the subspace of $\text{Spec}\ \C$ with the Zariski topology. 
It is readily seen that a map $u : | \C | \to \text{Spec}_r\ \C$ defined by 
$u(\varphi)=\text{Ker}\ \varphi$ is bijective. Moreover, the map $u$ is continuous. In fact, for an open base 
$D(f) = \{ {\sf m} \in  \text{Spec}_r\ \C \mid f \notin {\sf m}\}$ for some $f \in \C$, we see that 
$u^{-1}(D(f)) = \widetilde{f}^{-1}(\R\backslash \{0\})$.  

\begin{prop}{\em (}cf. \cite[Remark, page 23]{GS} {\em )} \label{prop:Spec}
The bijection $u : | \C | \stackrel{\cong}{\to} \text{\em Spec}_r\ \C$ is a homeomorphism.  
\end{prop}

\begin{proof}
With the same notation as in Lemma \ref{lem:2.1}, we see that 
$u(\widetilde{f_x}^{-1}(\R^+))=  D(f_x)$. We observe that $\widetilde{f_x}$ is non-negative since $\widetilde{f_x} \circ \theta = f_x$ with 
$\theta$ the bijection. 
\end{proof}

In consequence, the space $\text{Spec}_r\ \C$ is homeomorphic to $| \C |$ and hence the underlying space $S$: 
$$
S \cong | \C | \cong \text{Spec}_r\ \C \subset \text{Spec}\ \C. 
$$

\begin{rem} In \cite[4.3]{J} and \cite{D}, the spectrum for an $\mathbb{R}$-algebra corresponds to what we call the real spectrum of 
the $\mathbb{R}$-algebra, which in general is not the same as its prime spectrum.  
\end{rem}

The following result yields Theorem \ref{thm:main}. 

\begin{thm}\label{thm:key1}
 The forgetful functor $F : \mathsf{Stfd} \to {\mathbb R}\text{-}\mathsf{Alg}$ defined by $F(S, \C) = \C$ is fully faithful; that is, the induced map 
$F : \text{\em Hom}_ {\mathsf{Stfd}}((S, \C), (S', \C')) \to \text{\em Hom}_ {{\mathbb R}\text{-}\mathsf{Alg}}(\C', \C)$
 is a bijection.
\end{thm}

\begin{proof} For a morphism $\varphi : (S, \C) \to (S', \C')$ of stratifolds, namely a morphism of continuous spaces, we have a commutative 
diagram
$$
\xymatrix@C35pt@R20pt{
S \ar[r]^-\theta_-\cong \ar[d]_-{\varphi} & |\C| \ar[d]^-{|\varphi^*|} \\
S' \ar[r]^-{\theta}_-\cong &  \ |\C'|. 
}
\eqnlabel{add-1}
$$
In fact, we see that for any $f' \in \C'$ and $p \in S$, 
\begin{eqnarray*}
(|\varphi^*|\circ \theta)(p)(f') \!\!\! &=& \!\!\! |\varphi^*|(\theta(p))(f') = \theta (p)(\varphi^*(f')) \\ 
 &=& \theta (p)(f'\circ \varphi) = (f'\circ \varphi)(p) 
\end{eqnarray*}
and that $(\theta \circ \varphi)(p)(f') = \theta(\varphi(p))(f') = f'(\varphi(p))$. 
This yields that $F$ is injective. 

For any morphism $u : \C' \to \C$ of $\R$-algebras, we define $\varphi : S \to S'$ to be the composite 
$$
\xymatrix@C35pt@R18pt{
S \ar[r]^-\theta_-{\cong} & |\C| \ar[r]^-{|u|} & |\C'| \ar[r]^{\theta^{-1}}_-{\cong} & S'.  
}
$$
Observe that $|u|$ is a continuous map defined by $|u|(p) = p\circ u$; see \cite[3.19]{N}. 
For any $x \in |\C|$, 
we see that $|u|^*(\widetilde{f})(x) = (\widetilde{f}\circ |u|)(x) = \widetilde{f}(x\circ u) = (x\circ u)(f) = \widetilde{u(f)}(x)$. 
Thus it follows that $|u|^* : \widetilde{\C'} \to \widetilde{\C}$ is well defined. Moreover, we have a commutative diagram 
$$
\xymatrix@C35pt@R18pt{
\widetilde{\C'}  \ar[r]^-{|u|^*}  \ar[d]_-{\theta^*} & \widetilde{\C} \ar[d]^-{\theta^*} \\
\C' \ar[r]_-{u}&  \C. 
}
\eqnlabel{add-2}
$$
This follows from the fact that for any $\widetilde{f'} \in \widetilde{\C}$ and $x \in S$, 
\begin{eqnarray*}
(\theta^*\circ|u|^*)(\widetilde{f'})(x) \!\!\!&=& \!\!\! (|u|\circ \theta)^*(\widetilde{f'})(x) = (\widetilde{f'}\circ(|u|\circ \theta))(x) = 
((|u|\circ \theta)(x))(f') \\
&=& \!\!\!(|u|(\theta(x)))(f') = \theta(x)(u(f')) = u(f')(x). 
\end{eqnarray*}
Furthermore, we see that $(u\circ \theta^*)(\widetilde{f'})(x) = u(\theta (\widetilde{f'}))(x) = (u(\widetilde{f'}\circ \theta))(x) = 
u(f')(x)$. This enables us to deduce that $\varphi^* = u$.  
It turns out that  $F$ is a bijection. 
\end{proof}

We conclude this section with comments concerning Theorems \ref{thm:main} and \ref{thm:key1}. 

\begin{rem}\label{rem:F}
A stratifold $(S, \C)$ is a differential space. Then it is readily seen that 
$$
\text{Hom}_ {\mathsf{Stfd}}((S, \C), (\R, C^\infty(\R))) = \C. 
$$
\end{rem}

\begin{rem}
The result \cite[7.19]{N} asserts that the category $\mathsf{Mfd}$ of manifolds is equivalent to the category 
of {\it smooth} $\R$-algebras, which is a 
full subcategory of the category ${\mathbb R}\text{-}\mathsf{Alg}$. 
Moreover, we have the embedding $j : \mathsf{Mfd} \to \mathsf{Stfd}$ as mentioned above. 
However, Theorem \ref{thm:main} is not an immediate consequence of these results. 
\end{rem}

\section{The structure sheaf of a stratifold}

The goal of this section is to give a sheaf-theoretical extension of Theorem \ref{thm:main}. 

Let $X$ be a space and $\C$ an ${\mathbb R}$-subalgebra of $C^0(X)$ the ${\mathbb R}$-algebra of 
real-valued continuous functions on $X$. Recall that for any open subset $U$ of $X$, an element $f \in C^0(U)$ is called 
{\it locally extendable} in $\C$ if for any element $x$ in $U$, there exist an open neighborhood $V_x$ 
of $x$ in $U$ and a function $g \in \C$ such that $f{|_{V_x}} = g{|_{V_x}}$. 
It is readily seen that the pair $(X, \calO_X)$ is a ringed subspace 
of $(X, C^0)$ of real-valued continuous functions, 
where $\calO_X(X) = \C$ and 
$\calO_X(U)$ is the ${\mathbb R}$-subalgebra of $C^0(U)$ consisting of locally extendable elements in $\C$ 
for any open subset $U$ of $X$. 
Such a ringed subspace $(X, \calO_X)$ is called a {\it ringed continuous space}.  

A map between the underlying spaces of ringed continuous spaces, 
which induces a well-defined map between global sections, gives rise to a morphism of ringed spaces. 
The proof is straightforward.  More precisely,  we have 

\begin{lem}\label{lem:morphisms} 
Let $(X, \calO_X)$ and $(Y, \calO_Y)$ be ringed continuous spaces. Let $f : X \to Y$ be a continuous map. 
Suppose that $f^\sharp(g) := g\circ f$ is in $\calO_X(X)$ for any $g \in \calO_Y(Y)$. Then $f^\sharp$ induces a well-defined 
morphism of sheaves 
$
f^\sharp| : \calO_Y \to f_*\calO_X. 
$
\end{lem}
  
Let $(\mathsf{RS})_{C^0}$ be the category of ringed continuous spaces 
on locally compact, Hausdorff spaces with countable basis whose morphisms are 
continuos maps between underlying spaces satisfying the assumption on Lemma \ref{lem:morphisms}. 
Observe that every stratifold $(S, \C)$ gives rise to a ringed continuous space $(S, \calO_S)$. 
Its sheaf of rings $\calO_S$ will be called the {\it structure sheaf} of the stratifolds $(S, \C)$, and is explicitly given by 
$\calO_S(U) = \C_U$ for an open set $U$ of $S$; see Remark \ref{rem:substratifold}. 

\begin{defn}\label{defn:fine sheaf}
A ringed continuous space  $(X, \mathcal{O}_X)$ is called {\it fine} if the sheaf of rings $\calO_X$ is fine; 
that is, if for every locally finite cover ${\mathcal U}$ of $X$, there exists a partition of unity into a sum of global sections 
$s_i \in \mathcal{O}_X(X)$ whose supports are subordinate to ${\mathcal U}$. 
\end{defn} 
We work with the full subcategory 
$\mathsf{f}(\mathsf{RS})_{C^0}$ of $(\mathsf{RS})_{C^0}$ consisting of fine ringed continuous spaces.  
We have seen in Remark \ref{rem:substratifold} that the category $\mathsf{Stfd}$ fully faithfully embeds into 
$\mathsf{f}(\mathsf{RS})_{C^0}$. Moreover, we have the following extensions of results in Section \ref{section:Proof}. 

\begin{prop} \label{prop:sheaf} {\em (i)} The functor $F$ which assigns global sections gives rise to fully and faithful 
embedding form the category $\mathsf{f}(\mathsf{RS})_{C^0}$ into ${\mathbb R}\text{-}\mathsf{Alg}$. \\
{\em (ii)} Let $(X, \mathcal{O}_X)$ be in $\mathsf{f}(\mathsf{RS})_{C^0}$. Then there exist functional homeomorphisms 
$$
X \cong |\mathcal{O}_X(X)| \cong \text{\em Spec}_r \mathcal{O}_X(X). 
$$
{\em (iii)} The category $\mathsf{f}(\mathsf{RS})_{C^0}$ is a full subcategory of $\mathsf{RS}$ the category of ringed spaces. 
\end{prop}

\begin{proof}
The proofs of Theorem \ref{thm:main} and Proposition \ref{prop:spectrum} yield those of (i) and (ii). 

Let $(f, \varphi) : (X,  \mathcal{O}_X) \to (Y,  \mathcal{O}_Y)$ be a morphisms of ringed spaces. 
In order to prove (iii), it suffices to show that $\varphi = f^\sharp$. We define a continuous map 
$g : X \to Y$ by $g = \theta^{-1} \circ |\varphi_Y | \circ \theta$. 
The commutative diagram (2.2) allows us to deduce that 
$g^\sharp  : \mathcal{O}_Y(Y) \to (f_*\mathcal{O}_X)(Y) = \mathcal{O}_X(X)$ is 
nothing but the map $\varphi_Y$.   
For any open set $U$ of $Y$, we consider a commutative diagram 
$$
\xymatrix@C35pt@R20pt{
 \mathcal{O}_Y(U) \ar[r]^-{\varphi_U,  g^\sharp |} & (f_*\mathcal{O}_X)(U) \\
 \mathcal{O}_Y(Y) \ar[u]^{i^\sharp} \ar[r]_-{\varphi_Y = g^\sharp} & (f_*\mathcal{O}_X)(Y) \ar[u]_{i^\sharp},  
}
$$
where $i : U \to Y$ denotes the inclusion. 
By applying the realization functor $| \ |$ to the diagram above, we see that $\varphi_U =  g^\sharp |$.  
In fact, $|i^\sharp|$ is the inclusion $i$ up to homeomorphism $\theta$; see the commutative diagram (2.1).  
Suppose that $g(x) \neq f(x)$ for some $x \in X$. 
Then there exists an open neighborhood $V_{f(x)}$ of $f(x)$ such that 
$g(x)$ is not in  $V_{f(x)}$. 
On the other hand, since the map 
$g^\sharp | = \varphi_{V_{f(x)}} : \mathcal{O}_Y(V_{f(x)}) \to (f_*\mathcal{O}_X)(V_{f(x)})$ 
is well defined, it follows that $V_{f(x)} \supset g(f^{-1}(V_{f(x)}))$ 
and hence $g(x)$ is in $V_{f(x)}$, which is a contradiction. We have 
$(f, f^\sharp) = (g, g^\sharp) = (f, \varphi)$. 
This completes the proof of (iii).  
\end{proof}

We recall the category $\mathsf{Csp}$ of continuous spaces; see Section \ref{section:Proof}. 
Let $S : \R\text{-}\mathsf{Alg} \to \mathsf{Csp}$ be the contravariant functor defined by 
$S\F = (|\F|, \widetilde{\F})$. By the definition of a morphism in $\mathsf{Csp}$, we see that 
the same maps $\Phi$ and $\Psi$ as in Section \ref{section3} below give bijections 
$$
\xymatrix@C30pt@R20pt{
\text{Hom}_{\mathsf{Csp}^{op}}(S\F, (X, \C))  \ar@<0.5ex>[r]^-{\Phi} & 
\text{Hom}_{\R\text{-}\mathsf{Alg}}(\F, F(X, \C))  \ar@<0.5ex>[l]^-{\Psi}, 
}
$$
where $F : \mathsf{Csp}^{op} \to   \R\text{-}\mathsf{Alg}$ is the forgetful functor defined by 
$F(X, \C) = \C$.  
In fact, for $g$ in $\mathsf{Csp}$, the map $\Phi(g)$ factors through $\C$ and hence $\Phi$ is well defined. 
For $\varphi : \F \to \C$ and $\widetilde{f} \in \widetilde{\F}$, we have 
$(|\varphi | \circ \theta)^*(\widetilde{f})= \varphi(f) \in \C$. This implies that $\Psi$ is well defined. 
Thus $S$ is the left adjoint of $F$. 
Let $U : \mathsf{Csp}  \to \mathsf{Top}$ be the forgetful functor which assigns a continuous space the underlying space.  
We define a functor $m :  \mathsf{f}(\mathsf{RS})_{C^0} \to \mathsf{Csp}$ by 
$m(X, \mathcal{O}) = (X, \mathcal{O}(X))$ for a fine ringed continuous space $(X, \mathcal{O})$. 
With these functors, 
we have a digram 
$$
\xymatrix@C35pt@R20pt{
                    \mathsf{Csp} \ar[r]^-{U}  \ar@<0.5ex>[rd]^-{F} & \mathsf{Top} \\
\hspace*{-1cm}(\mathsf{RS})_{C^0}  \supset \mathsf{f}(\mathsf{RS})_{C^0}  \ar[u]^m \ar[r]_F & {\mathbb R}\text{-}\mathsf{Alg} \ar@<0.5ex>[lu]^-{S} \ar[u]_{\text{Spec}_r ( \ )}\\  
                     \mathsf{Stfd} \ar[ur]_{F} \ar[u]^{l}&   
}
\eqnlabel{add-1}
$$
in which the upper square and the triangles except for the upper right-hand side one are commutative up to isomorphism; 
see Propositions \ref{prop:spectrum} 
and \ref{prop:Spec}. Proposition \ref{prop:sheaf} (i) yields that the functor 
$F : \mathsf{f}(\mathsf{RS})_{C^0} \to {\mathbb R}\text{-}\mathsf{Alg}$ gives rise to an equivalence of categories between 
$\mathsf{f}(\mathsf{RS})_{C^0}$ and its image, which is a full subcategory ${\mathbb R}\text{-}\mathsf{Alg}$. 
One might remember the same result in algebraic geometry as the fact that the category of affine schemes 
is equivalent to the category of commutative rings with the global section functor.  

\begin{rem}
An object in  $\mathsf{f}(\mathsf{RS})_{C^0}$ which comes from $\mathsf{Stfd}$ is a locally ringed space; that is, the ring of germs at 
each point is local. This follows from the definition of a stratifold and \cite[Theorem 1.8]{GS}.    
\end{rem}

Let $A$ be an $\R$-algebra and $U$ an open set of $\text{Spec}_r A$. We put $M_U:=\bigcap_{\m \in U}\m^c$, where 
$\m^c$ denotes the complement of $\m$.  
Then $M_U$ is a multiplicative set. We denote by $M_U^{-1}A$ the localization of $A$ with respect to $M_U$. 
Define the {\it structure sheaf} ${\widehat{A}}$ on $\text{Spec}_r A$ by 
the sheafification of the presheaf $U\leadsto M_U^{-1}A$. Observe that the sheaf $\widehat{A}$ is the inverse image of the affine 
scheme $(\text{Spec} A, {\widetilde{A}})$ of $A$ along 
the inclusion $\text{Spec}_r A\hookrightarrow \text{Spec} A$.

The following proposition asserts that a stratifold is indeed a restriction of an affine scheme. 

\begin{thm} \label{thm:affine_scheme} Let $(S, \mathcal{O}_S)$ be a fine ringed space which comes from a stratifold 
$(S, \C)$ 
and $i : \text{\em Spec}_r \mathcal{O}_S(S) \to \text{\em Spec } \!\mathcal{O}_S(S)$ the inclusion.  
Then $(S, \mathcal{O}_S)$ is isomorphic to $i^*(\text{\em Spec } \!\mathcal{O}_S(S), \widetilde{\mathcal{O}_S(S)})$ 
as a ringed space, where $(\text{\em Spec } \!\mathcal{O}_S(S), \widetilde{\mathcal{O}_S(S)})$ is the affine scheme 
associated with the ring $\mathcal{O}_S(S)$.  
\end{thm}

\begin{proof} We recall the homeomorphism $\theta : S \stackrel{\cong}{\to} |S|$ and $u : |S| \stackrel{\cong}{\to} \text{Spec}_r  \mathcal{O}_S(S)$ in Section \ref{section:Proof}.  Let $m$ be the composite $u\circ \theta$. 
Then we have 
$$
m(p) = (u\circ \theta)(p) = \text{Ker } \! \theta (p) = \{f \in \C \mid f(p) =0\}=:\m_p.
$$ 
In order to prove the theorem, it suffices to show that $(S, \mathcal{O}_S)$ is isomorphic to the structure sheaf 
$(\text{Spec}_r \mathcal{O}_S(S), \widehat{\mathcal{O}_S(S)})$. To this end, we construct an isomorphism from 
$\widehat{\mathcal{O}_S(S)}$ to $m_*\mathcal{O}_S$. 

For an open set $U$ of $\text{Spec}_r \mathcal{O}_S(S)$, we define 
$\alpha_U :  M_U^{-1}\mathcal{O}_S(S)  \to (m_*\mathcal{O}_S)(U)$ by 
$\alpha([f/s])= f\cdot \frac{1}{s}$. Observe that $s(p) \neq 0$ for each $p$ in $m^{-1}(U)$.  
This implies that $\alpha_U$ is well defined. 
We see that  $\alpha_U$ induces a morphism of presheaves.  Moreover,  
the morphism of presheaves gives rise to a morphism 
$\widehat{\alpha} : \widehat{{\mathcal{O}_S(S)}} \to m_*\mathcal{O}_S$ of sheaves. 
The natural map 
$$\alpha_p : \widehat{\mathcal{O}_S(S)}_{\m_p}:= \text{colim}_{\m_p \in V}\widehat{\mathcal{O}_S(S)}(V)= {\mathcal{O}_S(S)}_{\m_p} 
\to \text{colim}_{p \in U}\mathcal{O}_S(U)=:\C_p$$ 
defined by $\alpha([f/s]) = f_p \cdot (\frac{1}{s_p})$ is well defined. Here ${\mathcal{O}_S(S)}_{\m_p}$ denotes 
the localization of the ring ${\mathcal{O}_S(S)}$ at ${\m_p}$. 
In fact, if $s \in \m_p^c = \mathcal{O}_S(S)\backslash \m_p$, then $s(p) = \theta(p)(s) \neq 0$. 
Since $(S, \mathcal{O}_S(S))$ is a stratifold, it follows that $1/s \in \mathcal{O}_S(U)$ for some open set $U$ of $S$. 
This follows from the condition (4) in Definition \ref{defn:stratifold}; see the proof of \cite[Proposition 2.3]{Kreck}. 
Moreover, there exists a bump function at each $x \in S$. 
Thus the proof of \cite[Corollary 1.6]{GS} enables us to conclude that 
$\alpha_p$ is an isomorphism. It turns out that 
$\widehat{\alpha} = \amalg_{p\in S} \alpha_p$ and hence $\widehat{\alpha}$ is an isomorphism. 
We have the result. 
\end{proof}

\begin{rem}\label{rem:C^infty}
For a stratifold $(S, \C)$, we regard $\C$ as a $C^\infty$-ring; see \cite{J, D, M-R}. 
Theorem \ref{thm:affine_scheme} asserts that the structure sheaf $(S, \mathcal{O}_S)$ of 
$(S, \C)$ is a $C^\infty$-ringed space in the sense of Joyce \cite{J} and is isomorphic to the spectrum of the $C^\infty$-ring $\C$; 
see \cite[Definition 4.12]{J} for example.
\end{rem}

\section{Vector bundles and the Serre-Swan theorem for stratifolds}

Generalizing the notion of smooth vector bundle over a manifold, 
we define a vector bundle over a stratifold. 

\begin{defn}\label{defn:bundles}
Let $(S,\C_S)$ be a stratifold and $(E,\C_E)$ a differential space. A morphism of differential spaces 
$\pi:(E,\C_E)\rightarrow (S,\C_S)$ is a {\it vector bundle} over $(S,\C_S)$ if the following conditions are satisfied. 
\begin{enumerate}
\item $E_x:=\pi^{-1}(x)$ is a vector space over $\R$ for $x\in S$.
\item There exist an open cover $\{U_\alpha\}_{\alpha \in J}$ of $S$ 
and an isomorphism $\phi_\alpha:\pi^{-1}(U_\alpha)\rightarrow U_\alpha \times \R^{n_\alpha}$ of differential spaces 
for each $\alpha \in J$. Here $\pi^{-1}(U_\alpha)$ is regarded as 
a differential subspace of $(E,\C_E)$; see Remark \ref{rem:substratifold}, 
  and  
$U_\alpha \times \R^{n_\alpha}$ is considered the product of the substratifold $(U_\alpha,\C_{U_\alpha})$ of $(S,\C_S)$ and 
the manifold $(\R^{n_\alpha},C^\infty(\R^{n_\alpha}))$; see Section \ref{section:App}.

\item The diagram
\[
\xymatrix@C30pt@R15pt{
\pi^{-1}(U_\alpha)\ar[rr]^{\phi_\alpha} \ar[rd]_\pi&&U_\alpha\times \R^{n_\alpha} \ar[ld]^{pr_1}\\
&U_\alpha&}
\]
is commutative, where $pr_1$ is the projection onto the first factor. 

\item The composite $pr_2\circ \phi_\alpha|_{E_x}: E_x\rightarrow U_\alpha\times\R^{n_\alpha}\rightarrow\R^{n_\alpha}$ is a linear isomorphism, where $pr_2:U_\alpha\times\R^{n_\alpha}\rightarrow \R^{n_\alpha}$ denotes the projection onto the second factor.
\end{enumerate}
\end{defn}

We call a vector bundle $\pi : (E,\C_E)\rightarrow (S,\C_S)$ {\it bounded} if for the index set $J$ of the cover which gives the trivialization, the set of integer $\{n_\alpha\}_{\alpha \in J}$ is bounded. Observe that the integer $n_\alpha$ is constant 
on a connected component of $S$. 

Let $\pi_E:E\rightarrow S$ and $\pi_F: F\rightarrow S$ be vector bundles over a stratifold $S$, 
We define a {\it  morphism $\varphi:E\rightarrow F$ of bundles} to be a morphism of differential spaces from $E$ to $F$ 
such that $\pi_F \circ \varphi = \pi_E$ and the restrictions on each stalks $\varphi_x:E_x\rightarrow F_x$ 
are linear maps. We denote by $\VB$ the category of vector bundles over $(S,\C)$ of bounded rank.

\begin{defn}
Let $ \pi : (E,\C_E) \to (S, \C_S)$ be a vector bundle over a stratifold $(S, \C_S)$.  
A morphism of differential spaces $s:(U,\C_U)\rightarrow (E,\C_E)$ is called a {\it section} 
on $U$ if $\pi\circ s=id_{(U,\C_U)}$. We denote by $\Gamma(U,E)$ the set of all sections. 
\end{defn}

We observe that $\Gamma(U,E)$ is an $\calO_S(U)$-module through the identification 
$\C_U = \calO_S(U)$. 
Moreover, we have the following proposition. 

\begin{prop}\label{prop:module}
The assignment $\Gamma( \ ,E) : U\leadsto \Gamma(U,E)$ gives rise to an $\calO_S$-module. 
\end{prop}

\begin{proof} We begin by showing that the assignment gives rise to a set-valued sheaf.  
Let $i : U \to V$ be an inclusion between open sets $U$ and $V$ of $S$. Since $i$ is a morphism of stratifolds, 
restricting along $i$ takes a section on $V$ to a section on $U$.

 
Let $\{V_\gamma\}_\gamma$ be an open cover of an open set $U$ of $S$.   
Suppose that $\{s_\gamma\}_\gamma$ in $\prod_\gamma \Gamma (V_\gamma, E)$ satisfies the condition that  
$res^{V_\gamma}_{V_\gamma\cap V_{\gamma'}}(s_\gamma) = res^{V_{\gamma'}}_{V_\gamma\cap V_{\gamma'}}(s_{\gamma'})$ 
for any $\gamma$ and $\gamma'$. 
In the category $\mathsf{Top}$ of topological spaces, 
we have a section $s : U \to E$ with $res^U_{V_\gamma}(s) = s_\gamma$ for any $\gamma$. 
We need to verify that $s$ is a morphism of differential spaces. 
For any $x \in U$, there exists an open set $V_\gamma$ such that $x \in V_\gamma$.     
Since $s_\gamma$ is a morphism of differential spaces, 
it follows that  $(s^*(p))|_{V_\gamma} = s_\gamma^*(p) \in \C_{V_\gamma}$ for $p \in \C_E$. By the definition of $\C_{V_\gamma}$, we 
see that there exists an open neighborhood $W_\gamma$ of $x$ with $W_\gamma \subset V_\gamma$ such that 
$(s^*(p))|_{W_\gamma} = ((s^*(p))|_{V_\gamma})|_{W_\gamma} = p\circ s_\gamma |_{W_\gamma} = h|_{W_\gamma}$ 
for some $h \in \C_S$.  This enables us to conclude that  $s^*(p) \in \C_U$ and hence $\Gamma( \ ,E)$ is a sheaf.   

For $s$ and $t$ in $\Gamma(U, E)$, we define a section $(s+t)$ in $\mathsf{Top}$ by $(s+t)(x) = s(x) + t(x)$ for any $x \in U$. 
We show that $(s+t)$ is in $\Gamma(U, E)$. 
Let $s_\gamma$, $t_\gamma : V_\gamma \to \pi^{-1}(V_\gamma)$ 
be the restrictions of $s$ and $t$ to $V_\gamma$, respectively. 
We define $\widetilde{s_\gamma}  : V_\gamma \to V_\gamma \times \R^n$ 
and $\widetilde{t_\gamma} :  V_\gamma \to V_\gamma \times \R^n$ by 
$\phi_\gamma\circ s_\gamma$ and 
$\phi_\gamma\circ t_\gamma$, respectively.  Here $\phi_\gamma : \pi^{-1}(V_\gamma) \stackrel{\cong}{\to} V_\gamma \times \R^n$ denotes a local trivialization.   

\begin{assertion}\label{assertion:1}
Let $s_\gamma$ be in $\Gamma(V_\gamma, E)$. Then $s_\gamma : V_\gamma \to \pi^{-1}(V_\gamma)$ is a morphism of 
differential spaces. 
\end{assertion}

Thus $\widetilde{s_\gamma}$ and $\widetilde{s_\gamma}$ are morphisms of differential spaces. The projection 
$pr_2 : V_\gamma \times \R^n \to \R^n$ into the second factor is a morphism of stratifolds and so are 
$pr_2\circ \widetilde{s_\gamma}$ and $pr_2\circ \widetilde{t_\gamma}$. 
We see that $pr_2\circ \widetilde{s_\gamma} + pr_2\circ \widetilde{t_\gamma}$ is a morphism of stratifolds. 
Proposition \ref{prop:products} below yields 
a morphism $(s_\gamma+ t_\gamma)' : V_\gamma \to V_\gamma\times \R^n$ of stratifolds which fits into a commutative diagram 
\[
\xymatrix@C60pt@R20pt{
& V_\gamma \ar[rd]^-{1_{V_{\gamma}}} \ar[ld]_-{pr_2\circ \widetilde{s_\gamma} + pr_2\circ \widetilde{t_\gamma} \ \ \ \ } 
\ar[d]^{(s_\gamma+ t_\gamma)' }& \\
\R^n & V_\gamma \times \R^n \ar[l]^{pr_2} \ar[r]_{pr_1}& V_\gamma. 
} 
\]
Define $s_\gamma + t_\gamma : V_\gamma \to \pi^{-1}(V_\gamma)$ to be the composite 
$\phi_\gamma^{-1}\circ (s_\gamma+ t_\gamma)'$. 
Observe that $(s_\gamma + t_\gamma)(x) = (s+t)(x) = (s_{\gamma'} + t_{\gamma'})(x)$ for $x \in U_\gamma\cap U_{\gamma'}$. 
Since $\Gamma( , E)$ is a sheaf, it follows that there exists a unique extension $\widetilde{s+t} \in \Gamma(U, E)$ of 
$\{s_\gamma+t_\gamma \}_\gamma$. It is readily seen that $\widetilde{s+t} = s+t$. 

The same argument as above does work well to show that $sk$ defined by $sk(x) = k(x)s(x)$ is in $\Gamma(U, E)$ 
for $k \in \C_U$ and $s\in \Gamma(U, E)$. This completes the proof. 
\end{proof}

\begin{proof}[Proof of Assertion \ref{assertion:1}] Let $x$ be an element in $V_\gamma$. 
For any $\rho \in \C_{\pi^{-1}(V_\gamma)}$, by definition,  
there exists an open neighborhood $W_{s_\gamma(x)}$ of $s_\gamma(x)$ such that 
$\rho |_{W_{s_\gamma(x)}}= \overline{\rho} |_{W_{s_\gamma(x)}}$ for some $\overline{\rho}\in \C_E$.  
Thus we see that 
\begin{eqnarray*}
\rho\circ s_\gamma |_{s_\gamma^{-1}(W_{s_\gamma(x)})} = 
\rho |_{W_{s_\gamma(x)}}\circ s_\gamma |_{s_\gamma^{-1}(W_{s_\gamma(x)})} 
&=& \overline{\rho} |_{W_{s_\gamma(x))}}\circ s_\gamma |_{s_\gamma^{-1}(W_{s_\gamma(x)})}\\
&=& 
\overline{\rho}\circ s_\gamma |_{s_\gamma^{-1}(W_{s_\gamma(x)})}. 
\end{eqnarray*}
Since $s_\gamma :  V_\gamma \to E$ is a morphism of differential spaces, it follows that $\overline{\rho}\circ s_\gamma$ is in 
$\C_{V_\gamma}$. Then $\overline{\rho}\circ s_\gamma$ 
is a restriction of a map in $\C_S$ to an appropriate open neighborhood of $x$ and hence so is 
$\rho\circ s_\gamma$. We have the result. 
\end{proof}

We denote by $\calL_E$ the $\calO_S$-module of Proposition \ref{prop:module}. 

\begin{lem}\label{lem:sections}
Let $pr_1: S\times\R^n\rightarrow S$ be the product bundle over a stratifold $(S,\C)$. The map $e_i:S\rightarrow S\times\R^n$ defined by $e_i(x)=(x,\mathsf{e}_i)$ is a section of this bundle for $i=1,...,n$, where $\{\mathsf{e_1},...,\mathsf{e_n}\}$ is a canonical basis for $\R^n$.
\end{lem}

\begin{proof}
We prove that $e_i$ is a morphism of differential spaces. 
Suppose that $f$ is in $\C_{S\times\R^n}$; see Section 6. Then there are local retractions $r_x:U_x\rightarrow U_x\cap S^j$ and $r_y=id: U_y\rightarrow U_y$ such that $f|_{U_x\times U_y}=f(r_x\times r_y)$ for $x\in S^j$ and $y=\mathsf{e}_i \in \R^n$. This yields 
that $f\circ e_i|U_x=f\circ e_i\circ r_x$. Since the restriction map $e_i|_{S^j}: S^j\rightarrow S^j\times \R^n$ is smooth, it follows that 
the composite $f\circ e_i|_{S^j}:S^j\rightarrow S^j\times\R^n\rightarrow\R$ is also smooth. In consequence, 
we have $f\circ e_i\in \C$.  
\end{proof}

\begin{prop}\label{prop:transition_function}
The transition functions $g_{\alpha\beta} : U_\alpha\cap U_\beta \to {\text GL}_n(\R)$ are morphisms of stratifolds. 
\end{prop}

\begin{proof}
By the definition of the transition function, we see that $\phi_\beta\phi_\alpha^{-1}(x, v) = 
(x, g_{\alpha\beta}(x)v)$. It follows from Lemma \ref{lem:sections} that the composite 
\[
\xymatrix@C35pt@R10pt{
\psi_j : U_\alpha\cap U_\beta \ar[r]^-{e_j} & U_\alpha\cap U_\beta \times \R^n \ar[r]^-{\phi_\beta \phi_\alpha^{-1}} & 
U_\alpha\cap U_\beta \times \R^n \ar[r]^-{pr_2} &\R^n
}
\]
is a morphism of differential spaces. Therefore, for the well-defined map 
$\psi_j^* : C^\infty({\mathbb R}^n) \to \C_{U_\alpha\cap U_\beta}$, 
we see that $u_{ij}:=\psi_j^*(p_i)= p_i\circ \psi_j \in \C_{U_\alpha\cap U_\beta}$ and $g_{\alpha\beta}(x) = ( u_{ij}(x) )$, where 
$p_i : \R^n \to \R$ is the projection onto the $i$th factor. It turns out that $g_{\alpha\beta}$ is a morphism of stratifolds. In fact, 
for any $f \in \C_{GL_n(\R)}$, there exists a smooth map $\overline{f} : M_{nn}(\R)=\R^{n^2} \to \R$ whose restriction coincides with $f$. Then we have $g_{\alpha\beta}^*(f) (x) = fg_{\alpha\beta}(x) = f(u_{11}(x), u_{12}(x), ..., u_{nn}(x)) = \overline{f}(u_{11}(x), u_{12}(x), ..., u_{nn}(x))$. This yields that  $g_{\alpha\beta}^*(f)$ is in $\C_{U_\alpha\cap U_\beta}$. 
\end{proof}

\begin{prop}\label{prop:stratifold_bundles}
Let $\pi:(E,\C_E)\rightarrow (S,\C_S)$ be a vector bundle in the sense of Definition \ref{defn:bundles}. Then 
the differential space $(E,\C_E)$ admits a stratifold structure for which $\pi$ is a morphism of stratifolds.
\end{prop}

\begin{proof} Without loss of generality, we assume that there exists a countable trivialization. 
Indeed $S$ has a countable basis. Thus the existence of a countable basis of $E$ follows from the local triviality. 
Moreover, the local triviality allows us to deduce that $E$ is a Hausdorff space.   

Let $S^i$ be a stratum of $S$. Observe that $S^i$ is a manifold for each $i$. 
By virtue of Proposition \ref{prop:transition_function}, we see  
that $\pi^{-1}(S^i)$ is a manifolds and $\pi : \pi^{-1}(S^i) \to S^i$ is a smooth vector bundle. 
It remains to prove that for any $x\in S^i$, the inclusion $i : \pi^{-1}(S^i) \to E$ induces an isomorphism 
$i^* : C(E)_x \to C^\infty(\pi^{-1}(S^i))_x$. Suppose that $x$ is in $U_\alpha$ with $\phi_\alpha : \pi^{-1}(U) \stackrel{\cong}{\to} U_\alpha \times \R^n$ a trivialization.  Then we have a commutative diagram 
\[
\xymatrix@C35pt@R20pt{
(\C_E)_x \ar[r]^-{i^*} \ar[d]_{res^*}^{\cong}& C^\infty(\pi^{-1}(S^i))_x \ar[d]^{res^*}_{\cong} \\
(\C_{\pi^{-1}(U_\alpha)})_x \ar[r]^-{i^*} & C^\infty(\pi^{-1}(S^i \cap U_\alpha))_x \\
(\C_{U_\alpha \times \R^n})_{\phi_\alpha(x)} 
\ar[r]^-{(i\times 1_{\R^n})^*} \ar[u]^{\phi_\alpha^*}_{\cong} & \C(S^i\cap U_\alpha \times \R^n)_{\phi_\alpha(x)} 
\ar[u]_{\phi_\alpha^*}^{\cong} 
}
\]
The stratifold structure on $U_\alpha \times \R^n$ allows us to deduce that $(i\times 1_{\R^n})^*$ is an isomorphism. 
Then we see that the upper horizontal arrow $i^*$ is an isomorphism. 

The local triviality of the bundle implies the existence of a bump function. In fact, the existence is a local property. 
This completes the proof. 
\end{proof}

\begin{prop}\label{prop:1}
Let $(S,\C)$ be a stratifold and $(E,\pi)\in \VB$. Then the $\calO_S$-module $\calL_E$ is a locally free module.
\end{prop}
\begin{proof}
Let $(\{U_\alpha\},\{\phi_\alpha:\pi^{-1}(U_\alpha)\rightarrow U_\alpha\times\R^{n_\alpha}\})$ be a trivialization. Define sections $s_i\in\calL_E(U_\alpha)$ by $s_i=\phi_\alpha^{-1}\circ e_i|_{U_\alpha}$ for $i=1,...,n_\alpha$. Then these sections are bases of $\calL_E(U_\alpha)$, so there exists an isomorphism between $\calL_E(U_\alpha)$ and $\calO_S(U_\alpha)^{n_\alpha}$. This induces an isomorphism between $\calL_E|_{U_\alpha}$ and $\calO_S^{n_\alpha}|_{U_\alpha}$.
\end{proof}

For $f\in \text{Hom}_{\VB}(E,F)$, we define a map $f_\ast: \Gamma(U,E)\rightarrow \Gamma(U,F)$ by $f_\ast(s)=f\circ s$. Since $f_x:E_x\rightarrow F_x$ are linear maps, it follows that $f_\ast$ is a morphism of $\calO_S(U)$-modules. Thus $f_\ast$ gives rise to a morphism $\calL_f:\calL_E\rightarrow \calL_F$. 
Let $\Lfb(S)$ be the full subcategory of $\calO_S$-$\mathsf{Mod}$ consisting of locally free $\calO_S$-modules of bounded rank. Proposition \ref{prop:1} enables us to define a functor $\calL:\VB\rightarrow \Lfb(S)$. 
Our goal of this section is to verify that the global section functor is an equivalence of categories as well as the usual result in case of smooth manifolds.

\begin{thm}\label{thm:3} Let $(S, \C)$ be a stratifold. Then the global section functor 
\begin{equation}
\Gamma(S,-): \VB\rightarrow \Fgp(\C) \nonumber
\end{equation}
gives rise to an equivalence of categories, where $\Fgp(\C)$ denotes the category 
of finitely generated projective modules over $\C$. 
\end{thm}

We shall prove Theorem \ref{thm:3} by using the  result due to Morye \cite{Morye}, 
Proposition \ref{prop:module} and an equivalence between categories $\VB$ and $\Lfb(S)$ for a stratifold 
$(S, \C)$, which is proved below. 

\begin{lem}\label{lem:1}
Let $X$ be a topological space and $\{X_\alpha\}$ an open cover of $X$. Suppose $(X_\alpha,\C_\alpha)$ is a differential space for each $\alpha$.  Define $\C$ to be the subalgebra of $C^0(X)$ consisting of $f:X\rightarrow\R$ such that $f|_{X_\alpha}\in \C_\alpha$ for all $\alpha$. Then the pair $(X,\C)$ is a differential space.
\end{lem}
\begin{proof}
The proof is straightforward. We check that $\C$ is a locally detectable $\R$-algebra. Let $f$ be in $C^0(X)$. Assume further that, for each $x\in X$, there are an open neighborhood $U_x$ of $x$ and $h_x\in \C$ such that $f|_{U_x}=h_x|_{U_x}$. Since $h_x|_{X_\alpha}\in \C_\alpha$, $f|_{U_x\cap X_\alpha}=h_x|_{U_x\cap X_\alpha}$ and $\C_\alpha$ is locally detectable, it follows that $f|_{X_\alpha}\in \C_\alpha$ and hence $f\in \C$ by definition. 
\end{proof}
 
\begin{prop}\label{prop:2}
The functor $\calL: \VB\rightarrow \Lfb(S)$ is essentially surjective.
\end{prop}
\begin{proof}
If $\F\in \Lfb(S)$, then there is an open cover $\{U_\alpha\}$ and isomorphisms $\varphi_\alpha: \F|_{U_\alpha}\rightarrow \calO_S^
{n_\alpha}|_{U_\alpha}$. Put $U_{\alpha,\beta}:=U_\alpha\cap U_\beta$. We have 
a transition function $g_{\alpha\beta}:U_
{\alpha\beta}\rightarrow \text{GL}_n(\R)$ induced by the isomorphisms 
$\varphi_\alpha$ and $\varphi_\beta$. 
More precisely, consider the sequence of morphisms of $\calO_{U_{\alpha\beta}}$-modules
\[
\xymatrix@C35pt@R25pt{
\calO_S|_{U_{\alpha\beta}}\ar[r]^{in_j} &\calO_S^n|_{U_{\alpha\beta}}\ar[r]^{\varphi^{-1}_\beta}_-{\cong} & 
\F|_{U_{\alpha\beta}}\ar[r]^{\varphi_\alpha}_-{\cong} &\calO_S^n|_{U_{\alpha\beta}} \ar[r]^{p_i} &\calO_S|_{U_{\alpha\beta}}, 
}
\]
where $in_j$ and $p_i$ denote the inclusion into the $j$th factor and the projection onto the $i$th factor, respectively. 
We define $u_{ij} \in  \calO_S|_{U_{\alpha\beta}}(U_{\alpha\beta})$ by 
$u_{ij} = p_i\varphi_\alpha\varphi_\beta^{-1}in_j ({\bf 1})$ with unit ${\bf 1}$ in  
$\calO_S|_{U_{\alpha\beta}}(U_{\alpha\beta})=\calO(U_{\alpha\beta})$. Then $g_{\alpha\beta}$ is defined by 
$g_{\alpha\beta} (x) = (u_{ij}(x))$ for $x \in U_{\alpha\beta}$. 

For each $x\in U_\alpha\cap U_\beta\cap U_\gamma$, these transition functions satisfy 
the relation $g_{\alpha\beta}(x)g_{\beta\gamma}(x)=g_{\alpha\gamma}(x)$. This enables us to define 
a space $E$ by the quotient space $(\bigsqcup_\alpha U_\alpha\times \R^{n_\alpha})/\sim$, where 
the equivalence relation $
\sim$ is defined by $(x,\mathsf{v})\sim(y,\mathsf{w})$ if $x=y\in U_{\alpha\beta}$ and $\mathsf{v}=g_{\alpha\beta}(x)\mathsf{w}$. Let 
$\rho:\bigsqcup_\alpha U_\alpha\times \R^{n_\alpha}\rightarrow E$ be the canonical projection. 
Then we define a continuous map $\pi:E\rightarrow S$ by $\pi(\rho(x,\mathsf{v}))=x$. 
Since the restriction $\rho_\alpha: U_\alpha\times\R^{n_\alpha}\rightarrow \pi^{-1}(U_\alpha)$ 
is a homeomorphism, it gives a subalgebra $\C_\alpha$ of $C^0(\pi^{-1}(U_\alpha))$ which is naturally isomorphic to 
$\C_{U_\alpha\times\R^{n_\alpha}}$. By Lemma \ref{lem:1}, we have a differential space $(E,\C_E)$. 

If $f\in \C_S$, then $(f\circ \pi)|_{\rho_\alpha(U_\alpha\times\R^{n_\alpha})}\in \C_\alpha$ since the projection 
$U_\alpha\times\R^{n_\alpha}\rightarrow U_\alpha$ is a morphism of differential spaces; see Section 6. 
This implies that $f\circ\pi\in \C_E$ and hence the map $\pi: (E,\C_E)\rightarrow (S,\C_S)$ 
is a morphism of differential spaces. Moreover, we can see that the morphism $\pi$ 
is a vector bundle with trivializations $(\{U_\alpha\},\{\rho_\alpha\})$. 

We shall show that $\calL_E$ is isomorphic to $\F$. For $s\in \calL_E(U_\alpha)$, we define $\hat{s}\in \calO_S(U_\alpha)^{n_\alpha}$ 
by the composite 
$pr_2\circ \rho_\alpha^{-1}\circ s:U_\alpha\rightarrow\pi^{-1}(U_\alpha)\rightarrow U_\alpha\times\R^n\rightarrow\R^n$. 
Since $\psi_\alpha:\calL_E(U_\alpha)\rightarrow \calO_S(U_\alpha)^{n_\alpha}$ defined by $\psi_\alpha(s)=\hat{s}$ is an isomorphism, 
it  gives rise to an isomorphism $\psi_\alpha:\calL_E|_{U_\alpha}\rightarrow \calO_S^{n_\alpha}|_{U_\alpha}$. 
The definitions of $\psi_\alpha$ and $E$ allow us to deduce that 
$\varphi_\alpha^{-1}\circ \psi_\alpha =\varphi_\beta^{-1}\circ \psi_\beta$. Therefore, we have a morphisms of equalizers 
\[
\xymatrix@C35pt@R20pt{
\calL_F(U) \ar[r] & \prod_\alpha \calL_F(U_\alpha \cap U) \ar@<0.5ex>[r]^-{res^{\alpha}_{\alpha\beta}}
\ar@<-0.5ex>[r]_-{res^{\beta}_{\alpha\beta}} \ar[d]^{\varphi_\alpha^{-1}\psi_\alpha}_{\cong} 
& \prod_{\alpha\beta} \calL_F(U_\alpha \cap U_\beta \cap U) 
\ar[d]^{\varphi_\alpha^{-1}\psi_\alpha=\varphi_\beta^{-1}\psi_\beta}_{\cong} \\
\F(U) \ar[r] & \prod_\alpha \F(U_\alpha \cap U) \ar@<0.5ex>[r]^-{res^{\alpha}_{\alpha\beta}}
\ar@<-0.5ex>[r]_-{res^{\beta}_{\alpha\beta}} & \prod_{\alpha\beta} \F(U_\alpha \cap U_\beta \cap U) . 
}
\]
This yields that $\calL_E\cong \F$ as an $\calO_S$-module. 
Hence the functor $\calL$ is essentially surjective.
\end{proof}

\begin{prop}\label{prop:3}
The functor $\calL$ is fully faithful.
\end{prop}

\begin{proof}
Let $f$ and $g$ be morphisms from $(E,\pi_E)$ to $(F,\pi_F)$. Assume that $\calL_f=\calL_g$. 
Then for all sections $s\in\Gamma(U,E)$, we see that $f\circ s=g\circ s$. This implies that $f=g$. 

Suppose that $f:\calL_E\rightarrow\calL_F$ is a morphism in 
$\Lfb(S)$ and $\varphi_\alpha:\calL_E|_{U_\alpha}\rightarrow \calO_S^n|_{U_\alpha}$ and 
$\psi_\alpha:\calL_F|_{U_\alpha}\rightarrow \calO_S^m|_{U_\alpha}$ are trivializations which is induced by 
the given trivializations of $E$ and $F$; see Proposition \ref{prop:1}. Then we obtain the following commutative diagram
\[
\xymatrix@C35pt@R20pt{\calO_S^n|_{U_\alpha} \ar[r]^{\varphi_\alpha^{-1}}_\cong \ar[d]_{t_\alpha} & 
\calL_E|_{U_\alpha} \ar[d]^f \\ \calO_S^m|_{U_\alpha} & \calL_F|_{U_\alpha} . \ar[l]^{\psi_\alpha}_\cong}
\]
The morphism $t_\alpha$ induces a morphism $t_\alpha:U_\alpha\rightarrow \text{Mat}_{m,n}(\R)$ of stratifolds with such way of defining $g_{\alpha\beta}$ in  the proof of Proposition \ref{prop:2}. 
We define a map $\eta_\alpha:E|_{U_\alpha}\cong U_\alpha\times\R^n\rightarrow U_\alpha\times\R^m\cong F|_{U_\alpha}$ by 
$\eta_\alpha(x,\mathsf{v})=(x,t_\alpha(x)\mathsf{v})$. This map is a morphism of stratifolds since the restriction on each manifold 
$(U_\alpha\cap S^i)\times\R^n$ is smooth and for $l\in\C_{S\times\R^m}$, there are local retraction 
$r_x: U_x\rightarrow U_x\cap S^i$ and open set $V$ of $\R^n$ such that 
$l\circ\eta_\alpha|_{U_x\times V}=l\circ\eta_\alpha(r_x\times id_V)$. 
Then the maps $\eta_\alpha$ induce a morphism $\eta:E\rightarrow F$ with 
$\eta_\alpha = \eta |_{\pi^{-1}(U_\alpha)} : \pi_E^{-1}(U_\alpha) \to \pi_F^{-1}(U_\alpha)$. 
In fact, we have a commutative diagram
\[
\xymatrix@C35pt@R18pt{\calO_S^n|_{U_{\alpha,\beta}} \ar[rr]^{\varphi_\alpha\circ\varphi_\beta^{-1}} \ar[rd] \ar[ddd]_{t_\beta} &&\calO_S^n|_{U_{\alpha,\beta}} \ar[ddd]^{t_\alpha} \\
&\calL_E|_{U_{\alpha,\beta}} \ar[ru] \ar[d]^f & \\
&\calL_F|_{U_{\alpha,\beta}} \ar[rd] & \\
\calO_S^m|_{U_{\alpha,\beta}} \ar[ru] \ar[rr]_{\psi_\alpha\circ\psi_\beta^{-1}} & &\calO_S^m|_{U_{\alpha,\beta}} .}
\]
By the construction of $\calL_\eta$, it is readily seen that $\calL_\eta=f$ and hence the functor $\calL$ is full.
\end{proof}

Thanks to Propositions \ref{prop:2} and \ref{prop:3}, we see that the functor $\calL$ is an equivalence of categories. 
We shall prove that the category $\Lfb(S)$ is equivalent to the full subcategory of the category $\Gamma(S,\calO_S)$-$\Mod$ consisting of finitely generated projective modules, which is denoted by $\Fgp(\Gamma(S,\calO_S))$. 

Following Morye, we say that the {\it Serre-Swan theorem} holds for a locally ringed space $(X,\calO_X)$ if the global section functor 
induces an equivalence of categories between  $\Lfb(X)$ and $\Fgp(\Gamma(X,\calO_X))$. 
The following theorem then completes the proof of Theorem \ref{thm:3}. 

\begin{thm}\label{prop:Morye}\text{\em(Morye, \cite[Corollary 3.2]{Morye}) }
Let $(X,\calO_X)$ be a locally ringed space such that $X$ is a paracompact Hausdorff space of finite covering dimension, 
and $\calO_X$ is a fine sheaf of rings (cf. Definition \ref{defn:fine sheaf}). Then the Serre-Swan theorem holds for $(X,\calO_X)$.
\end{thm}

The structure sheaf $\calO_S$ of a stratifold $(S, \C)$ is fine and the underlying space $S$ is paracompact; 
see Remark \ref{rem:substratifold}(iii).  In order to prove Theorem \ref{thm:3}, 
it is thus sufficient to show that the covering dimension $\text{dim\,}S$ of $S$ is finite. 

\begin{thm}\label{thm1}
\cite[Proposition 5.1 in chapter 3]{Pears} 
Any $n$-dimensional paracompact manifold $M$ (without boundary)  has covering dimension $\dim M = n$.
\end{thm}

\begin{thm}\label{thm2}
\cite[Proposition 5.11 in chapter 3]{Pears}
Let $X$ be a normal space and $A$ and $B$ be subspaces of $X$ such that $X=A\cup B$. Then,
$\text{dim\,}X\leq \text{dim\,}A+\text{dim\,}B+1$.
\end{thm}

\begin{cor}
Any finite-dimensional stratifold has finite covering dimension. 
\end{cor}
\begin{proof}
By the definition of a stratifold, we see that $S=S^1\sqcup S^2\sqcup...\sqcup S^n$, where $S^i$ is a manifold of dimension $i$. Theorems \ref{thm1} and \ref{thm2} imply that $\text{dim\,}S<\infty$.
\end{proof}

By definition, it follows that $\Gamma (S, \mathcal{O}_S) = \mathcal{O}_S(S) = \C$. Thus Proposition \ref{prop:Morye} and the results above 
enable us to deduce the following corollary. 

\begin{cor}\label{cor:1}
Let $(S,\C)$ be a stratifold and $\calO_S$ the structure sheaf. Then the global sections functor
$\Gamma(S,-):\Lfb(S)\rightarrow \Fgp(\C)$ 
is an equivalence. 
\end{cor}

We are now ready to prove the main theorem in this section. 

\begin{proof}[Proof of Theorem \ref{thm:3}] 
Corollary \ref{cor:1}, Proposition \ref{prop:2} and \ref{prop:3} yield Theorem \ref{thm:3}
\end{proof}

\begin{rem}\label{rem:comparison}
Theorem \ref{thm:affine_scheme} states that a stratifold $(S, \C)$ can be regarded as a subsheaf of an affine scheme 
of the form $\text{Spec } \!\mathcal{O}_S(S)$. Since the space $\text{Spec } \!\mathcal{O}_S(S)$ is compact, it follows that 
the real spectrum $\text{Spec}_r \mathcal{O}_S(S)$ is a proper subspace of the prime spectrum if $S$ is non-compact; see 
Proposition \ref{prop:sheaf}. Moreover, in general, there exists a point in $\text{Spec } \!\mathcal{O}_S(S)$ which is a maximal ideal 
but not in the real spectrum. Such a point is called a {\it ghost}; see \cite[8.22]{N}. 
However, Theorem \ref{thm:3} and the original Serre-Swan theorem yield 
that the category $\mathsf{VBb}_{(S, \C)}$ is equivalent to 
$\mathsf{VBb}_{\text{Spec } \!\mathcal{O}_S(S)}$ the category of vector bundles over the affine scheme 
$\text{Spec } \!\mathcal{O}_S(S)$ via the category $\Fgp(\Gamma(S,\calO_S))$; 
see \cite[Corollary 3.1]{Morye} and \cite[Theorem 6.2]{Shfarevich} for example.  
\end{rem}

\section{A local characterization of morphisms of stratifolds} \label{section3}

In this section, we describe morphisms of stratifolds inside the category of diffeological spaces. 
On the way we obtain a characterization of them by local data. We use the terminology of the book 
\cite{IZ} for diffeology. 




Let $\mathsf{Diffeology}$ be the category of diffeological spaces; see \cite{IZ}. 
We define a functor $k :  \mathsf{Stfd} \to \mathsf{Diffeology}$ by 
$k(S, \C) = (S, \D_\C)$ and $k(\phi) = \phi$ for a morphism $\phi : S \to S'$ of stratifolds, where 
$$
\D_\C=\{u : U \to S \mid U : \text{open in}  \ {\R^q}, q \geq 0, \phi\circ u \in C^\infty(U) \ \text{for any}  \ \phi \in \C \}. 
$$
Observe that a plot in $\D_\C$ is a set map. The functor $k$ is faithful, but not full; that is,  for a continuous map 
$f : S \to S'$, it is more restrictive to be a morphism of stratifolds $(S, \C) \to (S', \C')$ than to be a morphism of 
diffeological spaces $(S, \D_\C) \to (S', \D_{\C'})$. 

We recall the fully faithful functor $\ell : \mathsf{Mfd}  \to \mathsf{Diffeology}$ 
defined in \cite[4.3]{IZ}; see also \cite[Theorem 2.3]{C-S-W}. 
For a diffeological space $(X, \D)$, the set $X$ admits a topology which is referred as the $D$-topology. 
More precisely, a subset $A$ of $X$ is open if and only if $p^{-1}(A)$ is open for any plot $p \in \D$. 
We denote by $T(X, \D)$ the topological space. It is readily seen that the assignment of a topological space to a diffeological space 
induces a functor $T:  \mathsf{Diffeology}  \to \mathsf{Top}$.  

For a topological space $Y$, we define a diffeological space  $D(Y) = (Y, \D_Y)$ in which the set of plots $\D_Y$ consists of all continuous 
maps $U \to Y$ for any open subset $U$ of $\R^q$ and for $q \geq 0$. 

Let $\D_\R$ be the standard diffeology on $\R$. For each diffeological space $(X, \D)$, we have an $\R$-algebra
$F'((X, \D)):=\text{Hom}_{\mathsf{Diffeology}}((X, \D), (\R, \D_\R))$ with the algebra structure defined pointwise. 
A usual argument enables us to conclude that $F'$ gives rise to a contravariant functor 
$F':  \mathsf{Diffeology} \to {\mathbb R}\text{-}\mathsf{Alg}$. 

Summarizing the functors mentioned above, we have a diagram 
$$
\xymatrix@C35pt@R20pt{
\mathsf{Diffeology}  \ar@<0.5ex>[r]^-{T} \ar[ddr]^-{F'}& \mathsf{Top} \ar@<0.5ex>[l]^-{D}  \ar@<0.5ex>[dd]^-{C^0( \ )}\\
\mathsf{Stfd} \ar[u]_k \ar[rd]^-{F} & \\  
\mathsf{Mfd} \ar[u]_j \ar[r]_-{C^\infty ( \ )} \ar@/^2.0pc/[uu]^\ell& {\mathbb R}\text{-}\mathsf{Alg}  \ar@<0.5ex>[uu]^-{| \ |} 
}
\eqnlabel{add-3}
$$ 
in which the lower triangle and the left-hand side diagram are commutative. We observe that the functor 
$T$ is a left adjoint to $D$; see \cite[Proposition 3.1]{S-Y-H}. Moreover, it follows that 
the functor $C^0$ and $| \ |$ are adjoints. In fact, we have bijections
$$
\xymatrix@C30pt@R18pt{
\text{Hom}_{\mathsf{Top}}(X, |\F|)  \ar@<0.5ex>[r]^-{\Phi} & 
\text{Hom}_{\R\text{-}\mathsf{Alg}}(\F, C^0(X))  \ar@<0.5ex>[l]^-{\Psi}
}
$$
which are defined by the composites 
$\Phi(g) : \F \stackrel{\tau}{\to} \widetilde{\F} \stackrel{l}{\to} C^0(|\F|) \stackrel{g^*}{\to} C^0(X)$ with the inclusion $l$ 
and $\Psi(\varphi) : X \stackrel{\theta}{\to} |C^0(X)| \stackrel{|\varphi |}{\to} |\F|$, respectively. 
The bijectivity follows from a straightforward computation. 

We give here a characterization of morphisms of stratifolds in $\mathsf{Diffeology}$ with local data. 

\begin{prop}\label{prop:admissible_maps}
A morphism of diffeological spaces $f : (S, \D_\C) \to (S', \D_{\C'})$ stems from a morphism of stratifolds 
$f : (S, \C) \to (S', {\C'})$ 
if and only if for any $x \in S$, there exist 
local retractions $r_x : U_x \to U_x \cap S^i$ and $r_{f(x)} : V_{f(x)} \to V_{f(x)} \cap {S'}^j$ such that 
$r_{f(x)}\circ f\circ r_x = r_{f(x)}\circ f$ on some neighborhood of $x$. 
\end{prop}

\begin{proof} By definition, for any $x \in S$, there exists uniquely an integer $i$ such that $x$ is in $S^i$. 
Let $r_x : U_x \to U_x\cap S^i$ be a local retraction. The stratum $S^i$ is an $i$-dimensional manifold. Therefore, we have a local diffeomorphism 
$\varphi_i : V_x \to U_x\cap S^i$ for some open subset $V_x$ of $\R^i$. 
Let $u : V_x \to S$ be the composite $l\circ \varphi_i$, where $l : U_x\cap S^i \to S$ is the inclusion. 

Suppose that $f : (S, \D_\C) \to (S', \D_{\C'})$ is a morphism of diffeological spaces. 
In order to prove the ``if'' part, it suffices to show that for any $x \in S$, 
the induced morphism $f^* : \C_{f(x)}' \to {\mathcal Set}(S, {\mathbb R})_x$ factors through 
the $\R$-algebra $\C_x$ of germs, where ${\mathcal Set}(S, {\mathbb R})_x$ denotes the germ at $x$ of set maps $S \to {\mathbb R}$ associated with open neighborhoods of $x$. 
In fact, it follows that for any $\alpha \in \C'$, $f^*([\alpha]_{f(x)})= [\alpha\circ f]_x\in \C_x$. Then there exists $\beta \in \C$ such that $(\alpha \circ f) |_{W_x} = \beta|_{W_x}$ for some 
open subset $W_x$ of $S$.  Since the $\R$-algebra $\C$ is locally detectable,  we see that $f : S \to S'$ is a morphism of stratifolds and hence 
$k(f) = f$. 

Consider the following diagram 
$$
\xymatrix@C35pt@R18pt{
 & \C'_{f(x)} \ar[ld]_{(f\circ u)^*} \ar[rrd]^{f^*}& & \\
C^\infty(V_x)_{\varphi_i^{-1}(x)} \ar@<0.9ex>[r]^-{\psi_i^*}_-{\cong}& C^\infty(S^i)_x\ar@<0.9ex>[r]^-{r_x^*}_-{\cong} 
\ar@<1.2ex>[l]^-{\varphi_i^*} & \C_x \ar[r]_-s \ar@<1.2ex>[l]^-{l^*} & {\mathcal Set}(S, {\mathbb R})_x,    
}
$$
where $\psi_i$ is the local inverse of $\varphi_i$ and $s$ denotes the inclusion. 
Observe that $(f\circ u)^*  : \C'_{f(x)} \to C^\infty(V_x)_{\varphi_i^{-1}(x)}$ 
{\it is well defined since $f$ is a morphism of diffeological spaces}. 
For any $\alpha \in \C'$, we see that $\alpha = r_{f(x)}^*(\alpha)$ in $\C'_{f(x)}$; see \cite[page 19]{Kreck}. Thus it follows that 
\begin{eqnarray*}
r_{x}^*\psi_i^*(f\circ u)^*(r_{f(x)}^*(\alpha)) &=& \alpha \circ r_{f(x)} \circ f \circ l\circ \varphi_i \circ \psi_i \circ r_x \\
&=& \alpha \circ r_{f(x)} \circ f \circ r_x = \alpha \circ r_{f(x)}\circ  f = \alpha \circ f. 
\end{eqnarray*}
The third equality follows from the assumption. This implies that $f^*$ factors through the algebra $\C_x$ since 
$r_{x}^*\psi_i^*(f\circ u)^*(r_{f(x)}^*(\alpha))$ is in $\C_x$. 

We prove the ``only if" part. Let $f : (S, \C) \to (S', \C')$ be a morphism of stratifolds and 
$r_x : U_x \to U_x \cap S^i$ and $r_{f(x)} : V_{f(x)} \to V_{f(x)} \cap {S'}^j$ are appropriate local retractions. Without loss of generalities, 
we may assume that the image of $r_{f(x)}$ is contained in a local coordinate $V_{f(x)}'$ of the manifold 
$V_{f(x)}\cap {S'}^j$.  We have $r_x^*\circ \psi_i^*\circ (f\circ u)^* = f^*$. 
Observe that $(f\circ u)^* = \varphi_i^*\circ l^*\circ f^*$ and that the target of $f^*$ is the algebra $\C_x$.  
Let $\pi_k$ be an element in $\C^{\infty}( V_{f(x)} \cap {S'}^j)_{f(x)}$ obtained by extending the composite 
$V_{f(x)}' \stackrel{\cong}{\to} V' \stackrel{t}{\to} \R^j \stackrel{pr_k}{\to} \R$ by a bump function at $f(x)$, 
where $V_{f(x)}' \stackrel{\cong}{\to} V'$ is the homeomorphism of the local coordinate, 
$t$ is the inclusion and $pr_k$ denotes the projection onto the $k$th factor.  Then for the element 
$r_{f(x)}^*(\pi_k) \in \C_{f(x)}'$, we have $r_x^*\circ \psi_i^*\circ (f\circ u)^*(r_{f(x)}^*(\pi_k))= f^*\circ r_{f(x)}^*(\pi_k)$. 
The same argument as above enables us to deduce that
$$
\pi_k\circ r_{f(x)}\circ f\circ r_x =  \pi_k \circ r_{f(x)}\circ f
$$
on some neighborhood $W_x$ of $x$ and hence $r_{f(x)}\circ f\circ r_x = r_{f(x)}\circ f$ on $W_x$. 
This completes the proof. 
\end{proof}

\begin{cor} Let $M$ be a manifold and $(S, \C)$ a stratifold. Then the functor $k : {\mathsf{Stfd}} \to {\mathsf{Diffeology}}$ 
induces a bijection 
$$
k_* : \text{\em Hom}_{{\mathsf{Stfd}}}((M, C^\infty(M)), (S, \C)) \stackrel{\cong}{\to} 
\text{\em Hom}_{{\mathsf{Diffeology}}}((M,\D_{C^\infty(M)}), (S, \D_\C)).  
$$
\end{cor}

In the rest of this section, we give a subcategory of ${\mathsf{Diffeology}}$ which is equivalent to ${\mathsf{Stfd}}$ as a category. 

\begin{defn}\label{defn:admissible_maps}  Let $(S, \C)$ and  $(S', \C')$ be stratifolds. 
A continuous map $f : S \to S' $ is $(\C, \C')$-{\it admissible} if for any $x \in S$, there exist local retractions $r_x$ 
and $r_{f(x)}$ near $x$ and $f(x)$, respectively such that $r_{f(x)}\circ f \circ r_x = r_{f(x)}\circ f$ and for each $\phi \in \C'$, 
the restriction of $\phi \circ f : S \to {\mathbb R}$ to any stratum $S^i$ of $(S, \C)$ is smooth.   
\end{defn}

The proof of Proposition \ref{prop:admissible_maps}
yields the following result, which recovers as special case \cite[Exercise 2.6(11)]{Kreck}. 

\begin{prop}\label{prop:all_mor_admiaasible} A continuos map $f : S \to S'$ induces a morphism of stratifolds 
$(S, \C) \to (S', {\C'})$ if and only if $f$ is $(\C, \C')$-admissible. 
\end{prop}

Let $(S, \C)$, $(S', \C')$ and $(S'', \C'')$ be stratifolds. 
Proposition \ref{prop:all_mor_admiaasible} immediately implies that $(\C, \C')$-admissible continuous maps 
compose with  $(\C', \C'')$-admissible continuous maps $S' \to S''$.  

Let $k : \mathsf{Stfd} \to {\mathsf{Diffeology}}$ be the functor in (5.1) and $\langle  \mathsf{Im} k \rangle$ 
the full subcategory of $\mathsf{Diffeology}$ consisting of objects which come from $\mathsf{Stfd}$ by $k$. 
By the argument above, we have a wide subcategory $\langle  \mathsf{Im} k \rangle_W$ of 
$\langle  \mathsf{Im} k \rangle$ consisting of admissible maps and the same class of objects 
as in $\langle  \mathsf{Im} k \rangle$. 
Then Proposition \ref{prop:admissible_maps} establishes the following theorem. 

\begin{thm}\label{thm:a_highlight}
 The functor $k : \mathsf{Stfd} \to {\mathsf{Diffeology}}$ induces an equivalence 
 $k : \mathsf{Stfd} \to \langle  \mathsf{Im} k \rangle_W$ of categories.
In particular, one has a natural bijection   
\[
k_* : \text{\em Hom}_{{\mathsf{Stfd}}}((S, \C), (S', \C')) \stackrel{\cong}{\to} 
\text{\em Hom}_{\langle{\mathsf{Im}}k \rangle_W}((S, \D_\C), (S', \D_{\C'})).  
\]
\end{thm}

\section{Cartesian product of stratifolds}\label{section:App}

We recall the product of stratifolds defined in \cite{Kreck}. 
Let $(S,\C_S)$ and $(S',\C_{S'})$ be stratifolds. We define a stratifold with the underlying topological space $S\times S'$. 
Let $\C_{S\times S'}$ be the $\R$-algebra consisting of functions $f:S\times S'\rightarrow \R$ which are smooth on every products $S^i\times (S')^j$ and  for each $(x,y)\in S^i\times (S')^j$, there are local retractions $r_x: U_x \rightarrow S^i\cap U_x$ and $r_y: V_y\rightarrow (S')^j\cap V_y$ for which $f|_{U_x\times V_y}=f(r_x\times r_y)$. Then $(S\times S',\C_{S\times S'})$ is a stratifold and the projections into first and second factors are morphisms of stratifolds; see \cite[Appendix A]{Kreck}.

\begin{prop}\label{prop:products}
The product of stratifolds mentioned above is the cartesian product in the category $\mathsf{Stfd}$.
\end{prop}

We use lemmas to prove Proposition \ref{prop:products}.
\begin{lem}\label{lem:products}
Let $f_1 : (S_1, \C_1) \to (S_1', \C_1')$ and  $f_2 : (S_2, \C_2) \to (S_2', \C_2')$ be morphisms of stratifolds. 
Then the product of maps 
\[
f_1\times f_2 : (S_1\times S_2, \C_{S_1\times S_2}) \to (S_1'\times S_2', \C_{S_1'\times S_2'})
\] is a morphism of stratifolds. 
\end{lem}

\begin{proof}
For $x \in S_i$, assume that $x \in S_i^{k_i}$ and $f_i(x) \in {S'_i}^{j_i}$. 
Then we have a diagram 
\[
\xymatrix@C35pt@R18pt{
\C_{f_i(x)}' \ar[r]^-{\cong} \ar[d]_{f_i^*} & C^\infty({S'_i}^{j_i})_{f_i(x)}\\
\C_{x} \ar[r]^-{\cong} & C^\infty({S_i}^{k_i})_{x}\
}
\]
in which horizontal maps induced by the inclusions are isomorphisms. Therefore, for a smooth map $\varphi$ 
defined on an appropriate neighborhood of $f_i(x)$ in ${S'_i}^{j_i}$, 
we see that $\varphi \circ r_{f_i(x)} \circ f_i$ is a smooth map on some neighborhood of 
$x$ in $S_i^{k_i}$, where $r_{f_i(x)}$ denotes a local retraction near $f_i(x)$.
Thus, we infer that for any $h$ in $\C_{S_1'\times S_2'}$, 
\begin{eqnarray*}
h\circ (f_1\times f_2)|_{S_1^{k_1} \times S_2^{k_2}} &=& h\circ(r_{f_1(x_1)}\times r_{f_2(x_2)})\circ (f_1\times f_2) \\
&=& (h\circ \varphi_\alpha^{-1}) \circ (\varphi_\alpha \circ (r_{f_1(x_1)}\times r_{f_2(x_2)})\circ (f_1\times f_2)
\end{eqnarray*}
on some neighborhood of $(x_1, x_2)$ in $S_1^{k_1} \times S_2^{k_2}$, where $\varphi_\alpha$ is 
a local coordinate around $(f_1(x_1), f_2(x_2))$ of the manifold ${S_1'}^{j_1} \times {S_2'}^{j_2}$. 
This implies that $h\circ (f_1\times f_2)|_{S_1^{k_1} \times S_2^{k_2}}$ is smooth. Since 
$f_1$ and $f_2$ are admissible, it follows that for $h \in \C_{S_1'\times S_2'}$, 
\begin{eqnarray*}
h\circ (f\times f_2) \circ (r_{x_1}\times r_{x_2}) 
&=& h\circ (r_{f_1(x_1)}\times r_{f_2(x_2)}) \circ (f\times f_2) \circ (r_{x_1}\times r_{x_2}) \\
&=& h\circ (r_{f_1(x_1)}\times r_{f_2(x_2)}) \circ (f\times f_2) = h\circ (f\times f_2)
\end{eqnarray*} 
on an appropriate neighborhood of $(x_1, x_2)$ in $S_1\times S_2$. 
This completes the proof.
\end{proof}

By the same argument as in the proof of Lemma \ref{lem:products}, we have the following lemma. 
\begin{lem}\label{lem:diagonal_map}
The diagonal map $\Delta : S \to S\times S$ is a morphism of stratifolds.  
\end{lem}

\begin{proof}[Proof of Proposition \ref{prop:products}] Let $(S, \C)$, $(S', \C')$ and $(Z, \C_1)$ be stratifolds. Let 
$f_1 : (Z, \C_1) \to (S, \C)$ and $f_2 : (Z, \C_1) \to (S', \C')$ be morphisms of stratifolds. 
It suffices to show that $(f_1\times f_2) \circ \Delta $ is a morphism 
of stratifolds. This follows from Lemmas \ref{lem:products} and \ref{lem:diagonal_map}
\end{proof}

\noindent
{\it Acknowledgements.} 
The authors are grateful to Dai Tamaki for precious and beneficial comments on our work.
They are also indebted to the referee for valuable suggestions and improvements. 
The second author thanks Takayoshi Aoki and Wakana Otsuka for considerable discussions on stratifolds.  
This research was  partially supported by a Grant-in-Aid for challenging Exploratory Research 16K13753 from Japan Society for the Promotion of Science.


\begin{thebibliography}{99}
\bibitem{C-S-W} J.D. Christensen, G. Sinnanmon and Enxin Wu, The $D$-topology for diffeological space, 
Pacific Journal of Math. {\bf 272} (2014), 87--110. 
%
\bibitem{D} E.J. Dubuc, $C^\infty$-schemes, Amer. J. Math., {\bf 103} (1981), 683--690.
%
\bibitem{G} A. Grinberg, Resolutions of $p$-stratifolds with isolated singularities, 
Algebr. Geom. Topol. {\bf 3}  (2003), 1051--1078
%
\bibitem{IZ} P. Iglesias-Zemmour, Diffeology, Mathematical Surveys and Monographs, 185, AMS, Providence, 2012. 
%
\bibitem{Jakob} M. Jakob, A bordism-type description of homology, Manuscripta Math. {\bf 96} (1998), 67--80.
%
\bibitem{J} D. Joyce, Algebraic Geometry over $C^\infty$-rings, preprint, 2012.  {\tt arXiv}:1001.0023, 2010.
%
\bibitem{BK} B. Kloeckner, 
Quelques notions d'espaces stratifi\'es, Institut Fourier Grenoble,  S\'emin. Th\'eor. Spectr. G\'eom. {\bf 26} (2008), 13--28.
%
\bibitem{Kreck} M. Kreck, Differential Algebraic Topology, From Stratifolds to Exotic Spheres, Graduate Studies in Math., 110, AMS, 2010.
%
\bibitem{M-R} I. Moerdijk and G.E. Reyes, Models for smooth infinitesimal analysis,
Springer-Verlag, New York, 1991.
%
\bibitem{Morye} A.S. Morye, Note on the Serre-Swan Theorem, Math. Nachr. {\bf 286} (2013), 272--278. 
%
\bibitem{GS} J.A. Navarro Gonz\'alez and J. B. Sancho de Salas, $C^\infty$-Differential Spaces, 
Lecture Notes in Mathematics, 1824. Springer-Verlag, Berlin, 2003.
%
\bibitem{N} J. Nestruev, Smooth manifolds and observables, Graduate Texts Math. 220, Springer-Verlag, New York, 2002. 
%
\bibitem{Pears} A.R. Pears, Dimension theory of general spaces,
Cambridge University Press, 2008.
%
%
\bibitem{P} L.E. Pursell, Algebraic structures associated with smooth manifolds, Thesis, Purdue University, 1952.  
%
%
\bibitem{Serre} J.-P. Serre, Faisceaux alg\'ebriques coh\'erents, Ann. of Math. {\bf 61} (1955), 197--278. 
%
\bibitem{Shfarevich} I.R. Shafarevich, Basic algebraic geometry 2: Schemes and Complex Manifolds, 
second edition, Springer-Verlag, Berlin, 1997.
%
\bibitem{S-Y-H} K. Shimakawa, K. Yoshida and T. Hraguchi, Homology and cohomology via enriched bifunctors,  
{\tt arXiv:1010.3336}.
%
\bibitem{Sik} R. Sikorski, Differential modules, Colloq. Math. {\bf 24} (1971), 45--79. 
%
\bibitem{S} J.M. Souriau, Groupes diff\'erentiels, Differential geometrical methods in mathematical physics,  
Lecture Notes in Math., 836, Springer, (1980), 91--128. 
%
\bibitem{Spivak} D.I. Spivak, Derived smooth manifolds, Duke Math. J. {\bf 153} (2010), 55--128.
%
\bibitem{Swan} R.G. Swan, Vector bundles and projective modules, Trans. Amer. Math. Soc. {\bf 105} (1962), 264--277. 
\end{thebibliography}
\end{document}